\documentclass{amsart}
\usepackage{amsmath,amssymb,amscd,latexsym}

\usepackage[all]{xy}

\usepackage{graphicx}

\usepackage[usenames, dvipsnames]{color}

\newcommand{\bi}{\begin{itemize}}
\newcommand{\ei}{\end{itemize}}
\newcommand{\bn}{\begin{enumerate}}
\newcommand{\en}{\end{enumerate}}

\newcommand{\bq}{\begin{equation}}
\newcommand{\eq}{\end{equation}}

\newcommand{\A}{{\mathbb{A}}}
\newcommand{\C}{{\mathbb{C}}}
\newcommand{\DD}{\Delta}

\newcommand{\ED}{E_{\Dh}}

\newcommand{\HD}{H_{\Dh}}

\newcommand{\BPD}{BP_{\Dh}}

\newcommand{\BPnD}{BP\langle n \rangle_{\Dh}}
\newcommand{\Q}{{\mathbb{Q}}}
\newcommand{\R}{{\mathbb{R}}}
\newcommand{\Z}{{\mathbb{Z}}}
\newcommand{\cl}{\mathrm{cl}}

\newcommand{\HdgE}{\mathrm{Hdg}_{E}}

\newcommand{\Sing}{\mathrm{Sing}_*}

\newcommand{\Hom}{\mathrm{Hom}}

\newcommand{\Manc}{\mathbf{Man}_{\C}}

\newcommand{\sPre}{\mathbf{sPre}}

\newcommand{\Tb}{\mathbf{T}}
\newcommand{\sPreTb}{\sPre(\Tb)}

\newcommand{\hosPre}{\mathrm{ho}\sPre}

\newcommand{\sS}{\mathbf{sS}}

\newcommand{\Ch}{{\mathcal C}}
\newcommand{\Dh}{{\mathcal D}}

\newcommand{\Fh}{{\mathcal F}}

\newcommand{\Gh}{{\mathcal G}}

\newcommand{\Vh}{{\mathcal V}}

\newcommand{\into}{\hookrightarrow}

\newtheorem{theorem}{Theorem}[section]
\newtheorem{lemma}[theorem]{Lemma}
\newtheorem{prop}[theorem]{Proposition}
\newtheorem{cor}[theorem]{Corollary}
\theoremstyle{definition}
\newtheorem{defn}[theorem]{Definition}

\newtheorem{example}[theorem]{Example}

\newtheorem{remark}[theorem]{Remark}

\begin{document}

\author{Gereon Quick}

\address{Department of Mathematical Sciences, NTNU, NO-7491 Trondheim, Norway}
\email{gereon.quick@math.ntnu.no}

\title{Unstable splittings in Hodge filtered Brown-Peterson cohomology}


\date{}

\begin{abstract}
We construct Hodge filtered function spaces associated to infinite loop spaces. For Brown-Peterson cohomology, we show that the corresponding Hodge filtered spaces satisfy an analog of Wilson's unstable splitting. As a consequence, we obtain an analog of Quillen's theorem for Hodge filtered Brown-Peterson cohomology for complex manifolds.   
\end{abstract}

\maketitle

\section{Introduction}

For a fixed prime number $p$, let $BP^*(-)$ denote Brown-Peterson cohomology and let $BP$ be the $\Omega$-spectrum representing $BP^*(-)$ (see \cite{bp} and \cite{quillenfgl}). 
For an integer $n \ge 0$, let $BP\langle n \rangle$ be the $\Omega$-spectrum representing the associated intermediate cohomology theory defined in \cite{wilson}. These spectra are connected via canonical maps $g_n \colon BP \to BP\langle n \rangle$ and $f_n \colon BP\langle n \rangle \to BP\langle n-1 \rangle$. 
The coefficient rings of $BP$ and $BP\langle n \rangle$ are isomorphic to polynomial algebras 
\[
BP^* = \Z_{(p)}[v_1, v_2, \ldots]~ \text{and}~ BP\langle n \rangle^* = \Z_{(p)}[v_1, \ldots, v_n]
\]
where the generator $v_i$ has degree $-2(p^i-1)$. The effect of $g_n$ on coefficients is to send the generators $v_i$ with $i \ge n+1$ to $0$, and $f_n$ sends $v_n$ to $0$.


Let $BP_k$ and $BP\langle n \rangle_k$ denote the $k$th spaces of the spectra $BP$ and $BP\langle n \rangle$, respectively. In \cite{wilson}, Wilson proved that, for $k\le 2(p^n+p^{n-1}+\cdots+p+1)$, there is a homotopy equivalence
\bq\label{splittingintro}
BP_k \xrightarrow{\approx} BP\langle n \rangle_k \times \prod_{j>n}BP\langle j \rangle_{k+2(p^j-1)},
\eq
and the product cannot be broken down further. 
This result has many important consequences. For example, it yields a proof of Quillen's theorem that, for a finite complex $X$, $BP^*(X)$ is generated as a $BP^*$-module by elements of non-negative degree (\cite[Theorem 5.7]{wilson}).

In \cite{hfc}, Hopkins and the author showed that, given any rationally even spectrum $E$, there is an associated Hodge filtered cohomology theory $\ED^*(*)(-)$ for complex manifolds represented by a presheaf of spectra $\ED$. These Hodge filtered cohomology theories are natural generalizations of analytic Deligne cohomology for complex manifolds. In particular, there is a Hodge filtered $BP$-theory represented by a presheaf of spectra $\BPD$. 
The purpose of this paper is to show that Wilson's splitting \eqref{splittingintro} of the spaces $BP_k$ induces a splitting of the spaces, i.e. simplicial presheaves, in the spectrum $\BPD$.

It is important to note that the splitting \eqref{splittingintro} does not exist on the level of spectra (see \cite[p. 817]{bjw}). Hence the first important step is to construct Hodge filtered spaces associated to $BP$. 
The presheaf of spectra $\BPD$ is actually given by a wedge sum of presheaves of spectra $\BPD(m)$, one summand for each integer $m$.  
For all integers $m$ and $k\ge 0$, we construct simplicial presheaves $BP_k(m)$ on the Grothendieck site of complex manifolds together with weak equivalences $\sigma_k \colon BP_k(m) \to \Omega BP_{k+1}(m)$ such that the presheaf of spectra $\{BP_k(m)\}_k$ built by the $BP_k(m)$ and $\sigma_k$ is equivalent to $\BPD(m)$. 
In fact, we will construct spaces $E_k(m)$ for any connective rationally even spectrum $E$. We call $E_k(m)$ the $m$th Hodge filtered function space associated to $E_k$. For the reader who is familiar with differential cohomology theories, we remark that these constructions are similar to how Hopkins-Singer define differential function spaces for smooth manifolds in \cite{hs}. 

The second important step is to show that a map $E_k \to F_n$ between spaces of rationally even spectra $E$ and $F$ induces a map $E_k(m)\to F_n(m)$ of simplicial presheaves. 
The analog of Wilson's splitting will then be a natural consequence of the constructions. 

Our main results are the following.

\begin{theorem}\label{mainthmintro}
Let $m$ and $n$ be integers with $n\ge 0$. For every $k \leq 2(p^n + \cdots + 1)$, there is a weak equivalence of simplicial presheaves 
\bq\label{Hodge5.4intro}
BP_k(m) \xrightarrow{\approx} BP\langle n \rangle _k(m) \times \prod_{j>n} BP\langle j \rangle_{k+2(p^j-1)}(m).
\eq
\end{theorem}

It follows, in particular, that, for every complex manifold $M$, the natural map 
\[
\BPD^k(m)(M) \to BP\langle n \rangle_{\Dh}^k(m)(M)
\]
is surjective for $k \leq 2(p^{n} + \cdots + 1)$.

Furthermore, let $I^k\langle n \rangle(m)$ be the subgroup of elements in $\BPD^k(m)(M)$ which can be written as a finite sum 
\begin{align*}
u=\sum_{i > n} v_{i,m}(u_i)
\end{align*}
with $u_i \in \BPD^{k+2(p^i-1)}(m)(M)$, $v_i \in BP^{-2(p^i-1)}$, and $v_{i,m}$ the induced map $\BPD^*(m)(M) \to \BPD^{*+2(p^i-1)}(m)(M)$. Then Theorem \ref{mainthmintro} has the the following consequence. 

\begin{theorem}\label{injectiveintro} 
Let $M$ be a complex manifold and $m$ and $n \ge 0$ be integers. The natural induced homomorphism 
\[
\BPD^k(m)(M)/I^k\langle n \rangle(m) \to BP\langle n \rangle_{\Dh}^k(m) (M)
\]
is an isomorphism for $k \leq 2(p^n + \cdots + 1)$ and injective for $k \leq 2(p^n + \cdots + 1) + 2$. 
\end{theorem}

For $n=0$, we have $BP\langle 0 \rangle^*=\Z_{(p)}$ and $BP\langle 0 \rangle_{\Dh}^k (m)(M) = \HD^k(M; \Z_{(p)}(m))$, where $\HD^k(M; \Z_{(p)}(m))$ denotes the $k$th analytic Deligne cohomology of $M$ with coefficients in $\Z_{(p)}(m)$. 
Hence Theorem \ref{injectiveintro} has the following special case which is an analog of \cite[Theorem 2.2]{totaro} for Deligne cohomology. 

\begin{cor}\label{quillen-totaro-intro}
Let $M$ be a complex manifold and $m$ an integer. The natural map
\[
\BPD^{k}(m)(M)/I^k\langle 0 \rangle(m) \to \HD^{k}(M; \Z_{(p)}(m))
\]
is an isomorphism for $k \le 2$ and injective for $k\le 4$. In particular, the quotient $\BPD^{k}(m)(M)/I^k\langle 0 \rangle(m)$ vanishes for negative $k$. 
\end{cor}

The fact that $\BPD^{k}(m)(M)/I^k\langle 0 \rangle(m)$ vanishes in negative degrees could also be reformulated as the following analog of Quillen's theorem. 

\begin{theorem}\label{wilson5.7intro}
Let $M$ be a complex manifold and $m$ an integer. Then $\BPD^*(m)(M)$ is generated as a $BP^*$-module by elements of non-negative degree.
\end{theorem}

In \cite{totaro}, Totaro showed that Brown-Peterson cohomology and the map in Corollary \ref{quillen-totaro-intro} (for $BP$ and $H$ instead of $\BPD$ and $\HD$) are very useful tools for the study of the cycle maps for smooth projective complex algebraic varieties from  Chow groups to singular cohomology and Deligne cohomology, respectively. In \cite{hfc} and \cite{aj}, Hopkins and the author used Hodge filtered complex cobordism to study cycles in algebraic cobordism. We are optimistic that Theorem \ref{injectiveintro} will play an important role in the study of various related cycle maps in the future. \\


The author would like to thank Mike Hopkins for very helpful conversations and suggestions. 


\section{Hodge filtered function spaces}

We start with a brief recollection of Eilenberg-MacLane spaces, the singular functor and Hodge filtered forms in the setting of simplicial presheaves. Then we will construct Hodge filtered function spaces.

\subsection{Simplicial presheaves}

Let $\Tb$ be the category $\Manc$ of complex manifolds and holomorphic maps. The Grothendieck topology defined by open coverings turns $\Tb$ into an essentially small site with enough points. 
We denote by $\sPre=\sPreTb$ the category of simplicial presheaves on $\Tb$, i.e., contravariant functors from $\Tb$ to the category $\sS$ of simplicial sets. Objects in $\sPre$ will also be called \emph{spaces}. Recall that sending an object $M$ of $\Tb$ to the presheaf of sets it represents defines a fully faithful embedding of $\Tb$ into the category of presheaves of sets on $\Tb$. Since any presheaf of sets defines an object in $\sPre$ of simplicial dimension zero, we can embed $\Tb$ into $\sPre$. On the other hand, every simplicial set $K$ defines a simplicial presheaf by sending every object to $K$. By abuse of notation, we denote this simplicial presheaf by $K$ as well.

We will consider $\sPre$ with the local projective model structure (see e.g. \cite{dugger}, \cite{dhi}). 
We will not discuss the specific properties of this model structure, but just recall that a map $\Fh \to \Gh$ in $\sPre$ is called a (local) weak equivalence if the induced map of stalks $\Fh_x \to \Gh_x$ is a weak equivalence in $\sS$ for every point $x$ in $\Tb$. Furthermore, a map $\Fh \to \Gh$ is called an objectwise fibration if $\Fh(X) \to \Gh(X)$ is a fibration in $\sS$ for every $X\in \Tb$. A map is a local projective fibration if it is an objectwise fibration and satisfies descent for all hypercovers in $\Tb$ (see \cite[Corollary 7.1]{dhi}). 
We denote the corresponding homotopy category of $\sPre$ by $\hosPre$.


A natural way to send a $CW$-complex into $\sPre$ is the singular functor which is defined as follows. 
Let $\DD^n$ be the standard topological $n$-simplex 
\[
\DD^n =\{ (t_0,\ldots, t_n) \in \R^{n+1} | 0 \leq t_j \leq 1, \sum t_j =1 \}.
\]
For topological spaces $Y$ and $Z$, the singular function complex $\Sing(Z,Y)$ is the simplicial set whose $n$-simplices are continuous maps
\[
f \colon Z \times \DD^n \to Y.
\]
We denote the simplicial presheaf  
\[
M \mapsto \Sing(M,Y)=:\Sing Y(M)
\]
on $\Manc$ by $\Sing Y$. For any $CW$-complex $Y$, the simplicial presheaf $\Sing Y$ is objectwise fibrant and satisfies descent for hypercovers in $\Manc$ by \cite[Theorem 1.3]{di} (see also \cite[Lemma 2.3]{hfc}). By \cite[Corollary 7.1]{dhi}, this implies that $\Sing Y$ is a fibrant object in the local projective model structure on $\sPre$.

Furthermore, for a simplicial set $K$, let $|K|$ be its geometric realization in the category of $CW$-complexes. By \cite[Proposition 2.4]{hfc}, the natural map
\[
K \to \Sing |K|
\]
is a weak equivalence of simplicial presheaves. Hence we can use the assignment $K \mapsto \Sing |K|$ as a natural fibrant replacement in $\sPre$ for simplicial presheaves coming from simplicial sets.



An important class of simplicial presheaves are Eilenberg-MacLane spaces. 
Let $\Ch^*$ be a cochain complex of presheaves of abelian groups on $\Tb$. 
For any integer $n$, we denote by $\Ch^*[n]$ the cochain complex given in degree $q$ by $\Ch^q[n]:=\Ch^{q+n}$. The differential on $\Ch^*[n]$ is the one of $\Ch^*$ multiplied by $(-1)^n$. 
The hypercohomology $H^*(M; \Ch^*)$ of an object $M$ of $\Tb$ with coefficients in $\Ch^*$ is the graded group of morphisms $\Hom(\Z_M, a\Ch^*)$ in the derived category of cochain complexes of sheaves on $\Tb$, where $a\Ch^*$ denotes the complex of associated sheaves of $\Ch^*$. We will denote by $K(\Ch^*, n)$ the Eilenberg-MacLane spaces, i.e., simplicial presheaf, associated to $\Ch^*[-n]$. 
The following result is a version of Verdier's hypercovering theorem due to Ken Brown. 

\begin{prop}\label{verdier} {\rm (\cite[Theorem 2]{brown})} 
Let $\Ch^*$ be a cochain complex of presheaves of abelian groups on $\Tb$. Then for any integer $n$ and any object $M$ of $\Tb$, one has a canonical isomorphism 
\[
H^n(M; \Ch^*)\cong \Hom_{\hosPre}(M, K(\Ch^*, n)).
\]
\end{prop}

\begin{example}
Let $\Omega_M^n$ denote the sheaf of holomorphic $n$-forms on a complex manifold $M$ and let $\Omega^*$ be the complex of presheaves on $\Manc$ that sends a complex manifold $M$ to the complex $\Omega_M^*(M)$. The inclusion of complexes $\C \into \Omega^*$ is a quasi-isomorphism and induces a weak equivalence of simplicial presheaves
\[
K(\C,k) \to K(\Omega^*, k).
\]
This implies that there is a natural isomorphism  
\[
H^k(M; \C)\cong \Hom_{\hosPre}(M, K(\Omega^*, k))
\]
for every $k$ and $M \in \Manc$. 
\end{example}


\subsection{Hodge filtration on forms}

For a complex manifold $M$, let $\Omega_M^*$ again denote the complex of sheaves of holomorphic forms on $M$. Let $\Vh_{*}$ be an evenly graded $\C$-vector space such that each $\Vh_{2i}$ is a finite dimensional complex vector space. For an integer $n$, we denote by $H^n(M; \Vh_*)$ the graded cohomology group 
\[
H^n(M; \Vh_{*}) := \bigoplus_i H^{n+2i}(M; \Vh_{2i}).
\]

For integers $m$ and $n$, let $F^{m+i}H^{n+2i}(M; \Vh_{2i})$ be the $(m+i)$th step in the Hodge filtration of $H^{n+2i}(M; \Vh_{2i})$. We will write 
\[
F^mH^n(M; \Vh_{*}) := \bigoplus_i F^{m+i}H^{n+2i}(M; \Vh_{2i})
\]
for the graded Hodge filtered cohomology groups of $M$.

We denote by $\Omega_M^{*\ge m}$ the complex of sheaves of holomorphic forms of degree at least $m$. 
Let 
\[
\Omega^*_M \to A^*_M~\text{and}~\Omega^{*\ge m}_M \to F^mA^*_M
\]
be resolutions by cohomologically trivial sheaves which are functorial in $M$. We can choose these resolutions in such a way as to fit into a commutative diagram 
\[
\xymatrix{
\Omega_M^* \ar[d] \ar[r] & \cdots  \ar[r] & \Omega^{*\ge m-1}_M \ar[d]  \ar[r] & \Omega^{*\ge m}_M \ar[d] \ar[r] & \cdots \\
A^*_M \ar[r] & \cdots \ar[r] & F^{m-1}A_M^* \ar[r] & F^mA_M^* \ar[r] & \cdots
}
\]
For example, we could use the Godemont resolution (\cite[\S 3.2.3]{hodge2}). 
Let $A^*$ and $F^mA^*$ be the associated presheaves of complexes on $\Manc$ defined by 
\[
A^* \colon M \mapsto A^*_M(M)~\text{and}~F^mA^* \colon M \mapsto F^mA^*_M(M).
\]
Note that even though $A^*$ and $F^mA^*$ are double complexes, we will only consider their total complexes.

We denote by $A^*(\Vh_{2i})[-2i]$ the presheaf of forms with coefficients in $\Vh_{2i}$ shifted by degree $2i$ and we will write 
\[
A^*(\Vh_*):=\bigoplus_i A^*(\Vh_{2i})[-2i].
\]
For an integer $m$, we define the complex of presheaves $F^mA^*(\Vh_*)$ to be 
\bq\label{defofFpA*graded}
F^mA^*(\Vh_*)):=\bigoplus_i F^{m+i}A^*(\Vh_{2i})[-2i].
\eq
For an integer $n$, let $K(A^*(\Vh_*), n)$ and $K(F^mA^*(\Vh_*), n)$ denote the associated Eilenberg-MacLane spaces. 
Note that \eqref{defofFpA*graded} induces isomorphisms 
\[
K(A^*(\Vh_{*}), n) \cong \bigvee_i K(A^*(\Vh_{2i}), n+2i)
\]
\[
K(F^mA^*(\Vh_{*}), n) \cong \bigvee_i K(F^{m+i}A^*(\Vh_{2i}), n+2i).
\]

Recall that $|\cdot|$ denotes the geometric realization functor which sends simplicial sets to $CW$-complexes. 
The simplicial presheaf $\Sing |K(\Vh_{*}, n)|$ represents the functor of cocycles with coefficients in $\Vh_*$, i.e., for every $M \in \Manc$, there is a natural isomorphism of abelian groups
\[
Z^n(M; \Vh_{*}) \cong \Hom_{\sPre}(M, \Sing |K(\Vh_{*}, n)|), 
\]
where we write $Z^n(M; \Vh_*):= \oplus_i Z^{n+2i}(M; \Vh_{2i})$ for the graded group of cocycles on $M$. 
Since $M$ is a representable presheaf, we have a natural bijection of sets 
\[
\Hom_{\sPre}(M, \Sing |K(\Vh_{*}, n)|) \cong \mathrm{Sing}_0 |K(\Vh_{*}, n)|(M).
\]
Moreover, since $M$ is a cofibrant object in the local projective model structure on $\sPre$, there are natural bijections 
\bq\label{pi0EMcocycles}
\Hom_{\hosPre}(M, \Sing |K(\Vh_{*}, n)|) \cong \pi_0(\Sing |K(\Vh_{*}, n)|(M)) \cong H^n(M; \Vh_*).
\eq

Since the canonical inclusion $\Vh_* \into A_M^*(\Vh_*)$ is a quasi-isomorphism of complexes of sheaves for every $M$, it induces a weak equivalence of simplicial presheaves $K(\Vh_{*}, n) \to K(A^*(\Vh_{*}), n)$. 
Together with \eqref{pi0EMcocycles}, this implies that there are natural bijections
 \[
 \begin{array}{rcl}
\Hom_{\hosPre}(M, \Sing |K(A^*(\Vh_{*}), n)|) & \cong & \pi_0(\Sing |K(A^*(\Vh_{*}), n)|(M))\\
 & \cong & H^n(A^*(M; \Vh_*)) \\
 & \cong & H^n(M; \Vh_*).
\end{array}
\]

If $M$ is a compact K\"ahler manifold, we even have natural bijections 
\[
\begin{array}{rcl}
\Hom_{\hosPre}(M, \Sing |K(F^mA^*(\Vh_{*}), n)|) & \cong & \pi_0(\Sing |K(F^mA^*(\Vh_{*}), n)|(M))\\
 & \cong & H^n(F^mA^*(M; \Vh_*)) \\
 & \cong & F^mH^n(M; \Vh_*).
\end{array}
\]


\subsection{Hodge filtered function spaces}\label{HFS}

 We will now define Hodge filtered function spaces. The idea is similar to the way that differential function spaces were defined for smooth manifolds in \cite{hs}.

Let $m$ and $n$ be integers and $\Vh_{*}$ an evenly graded complex vector space. 
Let $Y$ be a CW-complex and let $\iota \in Z^n(Y; \Vh_{*})$ by a cocycle on $Y$.
A cocycle corresponds to a map of CW-complexes 
\[
Y \to |K(\Vh_{*}, n)|
\]
and induces a map of simplicial presheaves on $\Manc$
\[
\Sing Y \to \Sing |K(\Vh_{*}, n)|.
\]

The canonical inclusion $\Vh_* \into A^*(\Vh_*)$ induces a map $K(\Vh_{*}, n) \to K(A^*(\Vh_{*}), n)$, and we can form the following diagram of simplicial presheaves 
\bq\label{firstdef}
\xymatrix{
 & \Sing Y \ar[d]^{\iota^*}\\
\Sing |K(F^mA^*(\Vh_{*}), n)| \ar[r] & \Sing |K(A^*(\Vh_{*}), n)|.}
\eq

\begin{defn}
We define the \emph{Hodge filtered function space $(Y(m), \iota, n)$} to be the homotopy pullback of \eqref{firstdef} in $\sPre$. 
\end{defn}

\begin{remark}
Note that $(Y(m), \iota, n)$ depends on $\iota$ only up to homotopy, i.e., if $\iota'$ is another cocycle which represents the same cohomology class as $\iota$, 
then $(Y(m), \iota, n)$ and $(Y(m), \iota', n)$ are equivalent.  
\end{remark}


\begin{remark}\label{remarkDeligne1}
For $Y=K(\Z, n)$ and $\Vh_*=\C$ concentrated in degree $0$, we recover analytic Deligne cohomology for complex manifolds in the following way. 
Let $\iota \colon K(\Z, n) \to K(\C, n)$ be the map that is induced by the $(2\pi i)^m$-multiple of the inclusion $\Z \subset \C$ ($i$ being a square root of $-1$). Then 
\[
K(\Z, n)(m) := (K(\Z, n)(m), \iota, n)
\] 
represents analytic Deligne cohomology in $\sPre$, i.e., for every $M \in \Manc$, there is a natural isomorphism 
\[
H^n_{\Dh}(M; \Z(m)) \cong \Hom_{\hosPre}(M, K(\Z, n)(m)).
\]
\end{remark}

\subsection{An alternative definition}

An equivalent way to construct $(Y(m), \iota, n)$ is the following.  
For a complex manifold $M$, let $Z^n(M\times \Delta^{\bullet}; \Vh_{*})$ be the simplicial abelian group whose group of $k$-simplices is given by $C^{\infty}$-$n$-cocycles on $M \times \Delta^k$ with coefficients in $\Vh_{*}$. We denote the corresponding simplicial presheaf 
\[
M\mapsto Z^n(M \times \Delta^{\bullet}; \Vh_{*})
\]
on $\Manc$ by $Z^n(- \times \Delta^{\bullet}; \Vh_{*})$. 
A cocycle $\iota$ determines a map of simplicial presheaves
\[
\Sing Y \to Z^n(- \times \Delta^{\bullet}; \Vh_{*}), ~ f \mapsto \iota^*f, 
\]
given by taking the pullback along $\iota$. 
Let $I$ denote the map given by integration of forms
\[
I \colon F^{m+i}A^{n+2i}(M; \Vh_{2i}) \to C^{n+2i}(M; \Vh_{2i}), ~ \eta \mapsto (\sigma \mapsto \int_{\Delta^{n+2i}}\sigma^*\eta).
\]
We can form a diagram of simplicial presheaves 
\bq\label{seconddef}
\xymatrix{
 & \Sing Y \ar[d]^-{\iota^*}\\
\Sing |K(F^mA^*(\Vh_{*}), n)| \ar[r]_-I & Z^n(- \times \Delta^{\bullet}; \Vh_{*}).}
\eq
The map $\Sing K(\Vh_{*}, n)(M) \to Z^n(M \times \Delta^{\bullet}; \Vh_{*})$ given by pulling back a fundamental cocycle in $Z^n(K(\Vh_{*}, n); \Vh_{*})$ is a simplicial homotopy equivalence (see e.g.\,\cite[Proposition A.12]{hs}). Hence the homotopy pullback of \eqref{firstdef} is homotopy equivalent to the homotopy pullback of \eqref{seconddef}.

For given $m$, $n$ and $\iota$, let us write $Y(m)$ for the homotopy pullback of \eqref{seconddef}. To make the construction more concrete, we describe the $0$-simplices of $Y(m)(M)$ for a complex manifold $M$.  
Since one can calculate homotopy pullbacks in $\sPre$ objectwise (see e.g. \cite[Proposition 2.7]{aj}), we can assume that $Y(m)(M)$ is the homotopy pullback of the diagram of simplicial sets
\[
\xymatrix{
 & \Sing Y(M) \ar[d]^-{\iota^*}\\
\Sing |K(F^mA^*(\Vh_{*}), n)|(M) \ar[r]_-I & Z^n(M\times \Delta^{\bullet}; \Vh_{*}).}
\]

A $0$-simplex of $Y(m)(M)$ is given by a triple 
\[
q \colon M \to Y, ~ \eta \in F^mA^n(M; \Vh_{*})_{\cl}, ~ h \in C^{n-1}(M; \Vh_{*})
\]
where $q$ is a continuous map and $\eta$ is a closed form such that $\delta h = \iota^*q - I(\eta)$, where $\delta$ denotes the differential in $C^{*}(M; \Vh_{*})$. 


\section{Hodge filtered function spaces and spectra}

Our main case of interest is the construction of spaces in the Hodge filtered spectra defined in \cite[\S 4]{hfc}. We will first define such Hodge filtered function spaces and then explain how maps between loop spaces induce maps between Hodge filtered function spaces.

\subsection{Spaces in Hodge filtered $\Omega$-spectra}

Let $m$ be an integer. 
Let $E$ be a connective $\Omega$-spectrum built by $CW$-complexes. We assume that $E$ is rationally even, i.e., $\pi_*E\otimes \Q$ is concentrated in even degrees, and finitely generated in each degree. 
For example, $E$ could be either $BP$ or $BP\langle n \rangle$.

Let $E_k$ be the $k$th space of $E$. By our assumption on $E$, we have $\pi_{*+k}E_k = \pi_*E$. 
We set $\pi_*E_{\C} = \pi_*E\otimes_{\Z} \C$ and let 
\[
\mu_{E_k} \colon \pi_{*+k}E_k \to \pi_*E_{\C}
\]
be the graded homomorphism defined in degree $2j$ by 
\[
\pi_{2j+k}E_k \to \pi_{2j}E\otimes_{\Z}\C, ~ x \mapsto x\otimes (2\pi i)^{j+m}.
\]
 
The homomorphism $\mu_{E_k}$ corresponds to a cohomology class $c_{E_k} \in H^k(E_k; \pi_*E_{\C})$ under the Hurewicz isomorphism 
\[
H^k(E_k; \pi_*E_{\C}) \cong \Hom(\pi_{*+k}E_k, \pi_*E_{\C}).
\]

Let 
\bq\label{iotaEk}
\iota_{E_k} \colon E_k \to |K(\pi_*E_{\C}, k)|
\eq
be the map which represents a cocycle in $Z^k(E_k; \pi_*E_{\C})$ whose cohomology class is $c_{E_k}$. 
The choice of such a map is unique up to homotopy.

The inclusion $\pi_*E_{\C} \into A^*(\pi_*E_{\C})$ induces a map of simplicial presheaves 
\[
K(\pi_*E_{\C}, k) \to K(A^*(\pi_*E_{\C}), k). 
\]
Composition with \eqref{iotaEk} defines a map in $\sPre$   
\[
\iota_{E_k} \colon E_k \to |K(A^*(\pi_*E_{\C}), k)|
\]
which we also denote by $\iota_{E_k}$.

We call $\iota_{E_k}$ an \emph{$m$-twisted fundamental cocycle of $E_k$}. 
We can form the diagram in $\sPre$
\bq\label{Efirstdef}
\xymatrix{
 & \Sing E_k \ar[d]^-{\iota_{E_k}^*}\\
\Sing |K(F^mA^*(\pi_*E_{\C}), k)| \ar[r] & \Sing |K(A^*(\pi_*E_{\C}), k)|.}
\eq
We write $(E_k(m), \iota_{E_k})$ for the homotopy pullback of \eqref{Efirstdef} in $\sPre$. Note that a different choice $\iota'_{E_k}$ of an $m$-twisted fundamental cocycle of $E$ yields a homotopy equivalent simplicial presheaf $(E_k(m), \iota'_{E_k})$. Therefore, we will usually drop $\iota_{E_k}$ from the notation and write $E_k(m)$ for $(E_k(m), \iota_{E_k})$.

\begin{defn}
We call $E_k(m)$ the \emph{$m$th Hodge filtered function space of $E_k$}. 
\end{defn}

The relationship between the spaces $E_k(m)$ and the spectra $\ED(m)$ defined in \cite[\S 4]{hfc} is summarized in the following theorem. 

\begin{theorem}\label{EDiso}
For each $m$, we can choose the cocycles $\iota_{E_k}$ such that $\{E_k(m)\}_k$ forms an $\Omega$-spectrum in the category of presheaves of spectra which is equivalent to the spectrum $\ED(m)$ of \cite[\S 4]{hfc}. For each $k$, the simplicial presheaf $E_k(m)$ represents Hodge filtered $E$-cohomology groups of degree $k$ and twist $m$ in $\hosPre$, i.e., for any $M \in \Manc$, there is a natural isomorphism  
\[
\Hom_{\hosPre}(M, E_k(m)) \cong \ED^k(m)(M).
\]
\end{theorem}
\begin{proof}
We need to show that we can choose cocycles $\iota_{E_k}$ such that they are compatible with the structure maps of $E$. This will show that the family of simplicial presheaves $\{E_k(m)\}_k$ forms a presheaf of spectra. 

Let 
\[
\mu_E \colon \pi_*E \to \pi_*E_{\C}
\]
be the graded homomorphism defined by multiplication by $(2\pi i)^{j+m}$ in degree $2j$. 
The homomorphism $\mu_E$ corresponds to a cohomology class $c_E \in H^0(E; \pi_*E_{\C})$ under the Hurewicz isomorphism
\[
H^0(E; \pi_*E_{\C}) \cong \Hom(\pi_*E, \pi_*E_{\C}).
\]

The class $c_E$ in $H^0(E; \pi_*E_{\C})$ can be represented by a map of spectra 
\[
\iota_E \colon E \to H(\pi_*E_{\C})
\]
where $H(\pi_*E_{\C})$ denotes the Eilenberg-MacLane spectrum associated to the graded $\C$-vector space $\pi_*E_{\C}$. 
This map consists of a family of maps  
\[
\iota_{E_k} \colon E_k \to |K(\pi_*E_{\C}, k)|
\]
which are compatible with the structure maps of the spectra $E$ and $H(\pi_*E_{\C})$. 
More precisely, if $\sigma_k \colon E_k \to \Omega E_{k+1}$ denotes the $k$th structure map of $E$, these cocycles induce commutative diagrams of the form
\[
\xymatrix{
E_k \ar[r]^-{\iota_{E_k}} \ar[d]_-{\sigma_k} & |K(\pi_*E_{\C}, k)| \ar[d] \\
\Omega E_{k+1} \ar[r]^-{\Omega\iota_{E_{k+1}}} & \Omega |K(\pi_*E_{\C}, k+1)|}
\]
where the right hand vertical map is the $k$th structure of $H(\pi_*E_{\C})$.

Hence, since $\sigma_k$ and $\iota_{E_k}$ are compatible, we obtain a commutative diagram
\[
\xymatrix{
\Sing E_k \ar[d]_{\sigma_k} \ar[r]^-{\iota_{E_k}} & \Sing |K(A^*(\pi_{*}E_{\C}), k)| \ar[d] & \ar[l] 
\Sing |K(F^mA^*(\pi_*E_{\C}),k)|  \ar[d] \\
\Sing \Omega E_{k+1} \ar[r]_-{\Omega\iota_{E_{k+1}}} & \Sing \Omega|K(A^*(\pi_*E_{\C}), k+1)| & \ar[l] 
\Sing \Omega|K(F^mA^*(\pi_*E_{\C}), k+1)|.}
\]

Since $E$ is an $\Omega$-spectrum, the vertical maps are weak equivalences in $\sPre$. Hence the map from the homotopy pullback of the upper row to the homotopy pullback of the lower row, which is induced by the vertical maps, is a weak equivalence in $\sPre$. Since taking loop spaces commutes with the singular functor and with taking homotopy pullbacks, this shows that the diagram above induces a map of simplicial presheaves 
\[
\sigma_k(m) \colon E_k(m) \to \Omega E_{k+1}(m)
\]
which is a weak equivalence in $\sPre$. 
This proves that the family of simplicial presheaves $E_k(m)$ together with maps $\sigma_k(m)$, indexed by $k$, forms a presheaf of spectra which we denote by $E(m)$. It follows from the construction of $E(m)$ and the presheaf of spectra $\ED(m)$ in \cite[\S 4.1]{hfc} that there is a canonical map of presheaves of spectra
\bq\label{Emap}
E(m) \to \ED(m).
\eq

Now let $M$ be a complex manifold. It follows from the definition of $E_k(m)$ as a homotopy pullback that the group $\Hom_{\hosPre}(M, E_k(m))$ sits in a long exact sequence analog to the one of \cite[Proposition 4.5]{hfc}. Moreover, the map \eqref{Emap} induces a natural morphism of long exact sequences
\bq\label{longexactsequence}
\scalebox{0.8}{
\xymatrix{
\cdots \ar[r] & E_{\C}^{k-1}(M) \ar[r] \ar[d] & \Hom_{\hosPre}(M, E_k(m))  \ar[r] \ar[d] & 
E^k(M)\oplus H^k(M; F^mA^*(\pi_{*}E_{\C})) \ar[r] \ar[d] & \cdots \\
\cdots \ar[r] & E_{\C}^{k-1}(M) \ar[r] & \ED^k(m)(M) \ar[r] & 
E^k(M)\oplus H^k(M; \Omega^{*\ge m}(\pi_{*}E_{\C})) \ar[r] & \cdots}}
\eq

Since the outer vertical maps are isomorphisms, the induced map 
\[
\Hom_{\hosPre}(M, E_k(m)) \xrightarrow{\cong} \ED^k(m)(M)
\]
is an isomorphism as well. Since we can assume that $\ED(m)$ is an $\Omega$-spectrum, this also implies that the map \eqref{Emap} is an objectwise and hence also stalkwise weak equivalence of spectra. 
\end{proof}

\begin{remark}\label{compactcase}
In order to indicate how the groups $\ED^k(m)(M)$ look like, let us assume $k=2m$ and that $M$ is a compact K\"ahler manifold. 
In this case, one can show using Hodge theory as in \cite[\S 4.3]{hfc} that the (upper) long exact sequence in \eqref{longexactsequence} can be split into short exact sequences of the form
\[
0 \to E^{2m-1}(M)\otimes_{\Z} \R/\Z \to \ED^{2m}(m)(M) \to \HdgE^{2m}(M) \to 0
\]
where $\HdgE^{2m}(M)$ is the subgroup of $E^{2m}(M)$ that is defined by the cartesian square 
\[
\xymatrix{
\HdgE^{2m}(M) \ar[d] \ar[r] & E^{2m}(M) \ar[d] \\ 
F^mH^{2m}(M; \pi_{*}E_{\C}) \ar[r] & H^{2m}(M; \pi_{*}E_{\C}).}
\]
\end{remark}

\subsection{Induced maps and products}

The construction of Hodge filtered function spaces is functorial in the following way.

\begin{prop}\label{inducedmap}
Let $m$ be an integer. Let $E$ and $F$ be connective rationally even $\Omega$-spectra and $f \colon E_k \to F_n$ be a map from the $k$th space of $E$ to the $n$th space of $F$. 

{\rm a)} Then $f$ induces a map of Hodge filtered function spaces 
\[
f(m) \colon E_k(m) \to F_n(m).
\]

{\rm b)} If $f$ is a weak equivalence, then $f(m)$ is a weak equivalence of simplicial presheaves. 
\end{prop}

\begin{proof}
a) We set $\pi_*E_{\C} := \pi_*E\otimes_{\Z}\C$ and $\pi_*F_{\C} := \pi_*F\otimes_{\Z}\C$. 
We define graded homomorphisms $\mu_{E_k} \colon \pi_{*+k}E_k \to \pi_*E_{\C}$ and $\mu_{F_n} \colon \pi_{*+n}F_n \to \pi_*F_{\C}$ by multiplication by $(2\pi i)^{j+m}$ for $*=2j$. 

The given map $f$ induces graded homomorphisms 
\[
\pi_{*}E_k \to \pi_{*}F_n ~ \text{and} ~ \pi_{*}E \to \pi_{*+(k-n)}F.
\]

On the level of Eilenberg-MacLane spaces, $f$ induces a map 
\[
f_K \colon K(\pi_*E_{\C}, k) \to K(\pi_*F_{\C}, n).
\]


Now let $c_{E_k}$ be the cohomology class in $H^k(E_k; \pi_*E_{\C})$ corresponding to $\mu_{E_k}$ under the Hurewicz isomorphism, and let  
\[ 
\iota_{E_k} \colon E_k \to |K(\pi_*E_{\C}, k)|
\]
be a cocycle in $Z^k(E_k; \pi_*E_{\C})$ whose cohomology class is $c_{E_k}$. 
Similarly, let $c_{F_n}$ be the cohomology class in $H^n(F_n; \pi_*F_{\C})$ corresponding to $\mu_{F_n}$ under the Hurewicz isomorphism, and let  
\[ 
\iota_{F_n} \colon F_n \to |K(\pi_*F_{\C}, n)|
\]
be a cocycle in $Z^n(F_n; \pi_*F_{\C})$ with cohomology class $c_{F_n}$. 
The images of $c_{E_k}$ and $c_{F_n}$ in $H^n(E_k; \pi_*F_{\C})$ under the maps 
\[
\xymatrix{
 & H^n(F_n; \pi_*F_{\C}) \ar[d]^-{f^*}\\
H^k(E_k; \pi_*E_{\C}) \ar[r]_-{f_{K*}} & H^n(E_k; \pi_*F_{\C})}
\] 
induced by $f$ and $f_K$ agree, since they are both equal to the class in $H^n(E_k; \pi_*F_{\C})$ which corresponds to the composed graded homomorphism $\pi_{*+k}E_k \to \pi_{*+(k-n)}F_{\C}$ under  
$H^n(E_k; \pi_*F_{\C}) \cong \Hom(\pi_{*+k}E_k, \pi_{*+(k-n)}F_{\C})$. 
This implies that the diagram
\bq\label{cocyclecommdiag}
\xymatrix{
E_k \ar[r]^-f \ar[d]_-{\iota_{E_k}} & F_n \ar[d]^-{\iota_{F_n}} \\
|K(\pi_*E_{\C}, k)| \ar[r]_-{f_K} & |K(\pi_*F_{\C}, n)|}
\eq
commutes up to homotopy. Hence there is a cochain $b_{F_n} \in C^{n-1}(F_n; \pi_*F_{\C})$ such that, after replacing $\iota_{F_n}$ with $\iota_{F_n}-\delta b_{F_n}$, the diagram commutes, i.e., we have $f_K\circ \iota_{E_k} = (\iota_{F_n}-\delta b_{F_n})\circ f$. 
Since the maps $\iota_{F_n}$ and $\iota_{F_n}-\delta b_{F_n}$ are homotopic, the simplicial presheaves $(F_n(m), \iota_{F_n}-\delta b_{F_n})$ and $(F_n(m), \iota_{F_n})$ are canonically weakly equivalent.

Hence for the construction of $F_n(m)$ we can assume from now on that the cocycle $\iota_{F_n}$ is chosen such that diagram \eqref{cocyclecommdiag} commutes. 
Then the map $f$ induces a commutative diagram of simplicial presheaves
\begin{equation}\label{mapdiag}
\xymatrix{
\Sing E_k \ar[d]_{\Sing f} \ar[r]^-{\iota_{E_k}} & \Sing |K(A^*(\pi_{*}E_{\C}), k)| \ar[d]_{\Sing f_K} & \ar[l] 
\Sing |K(F^mA^*(\pi_*E_{\C}), k)|  \ar[d]_{\Sing f_K} \\
\Sing F_n \ar[r]_-{\iota_{F_n}} & \Sing |K(A^*(\pi_*F_{\C}), n)| & \ar[l] 
\Sing |K(F^mA^*(\pi_*F_{\C}), n)|.}
\end{equation}

If we form the homotopy pullback of the top row, which is $E_k(m)$, and the homotopy pullback of the bottom row, which is $F_n(m)$, we obtain that $f$ induces a map of simplicial presheaves 
\[
f(m) \colon E_k(m) \to F_n(m).
\]

b) If $f\colon E_k \to F_n$ is a weak equivalence, then the vertical maps in diagram \eqref{mapdiag} are all weak equivalences. Hence the induced map of homotopy pullbacks $f(m)$ is also a weak equivalence. 
\end{proof}

\begin{remark}
Let $f \colon E_k \to F_n$ be a map as in Proposition \ref{inducedmap} and let $M$ be a complex manifold $M$. One should note that even though $f(m)$ may depend on the chosen cocycles, the induced map on Hodge filtered cohomology groups 
\[
f(m)_* \colon \Hom_{\hosPre}(M, E_k(m)) \to \Hom_{\hosPre}(M, F_n(m))
\]
only depends on the homotopy type of $f$.
\end{remark}

Now we show that Hodge filtered function spaces behave well under taking products. Let $E$ and $F$ be connective rationally even $\Omega$-spectra, and let $E_k$ and $F_n$ be their $k$th and $n$th spaces, respectively. We can choose, independently, cocycles $\iota_{E_k}$ and $\iota_{F_n}$ which represent the homomorphisms $\mu_E$ and $\mu_F$ as in the beginning of the proof of Proposition \ref{inducedmap}. 
Since we have
\[
\pi_*(E\times F) = \pi_*E \oplus \pi_*F,
\]  
there is a canonical homotopy equivalence
\[
|K(\pi_*(E\times F)_{\C}, k+n)| \xrightarrow{\approx} |K(\pi_*E_{\C}, k)| \times |K(\pi_*F_{\C}, n)|.
\]
Hence we can use $\iota_{E_k}$ and $\iota_{F_n}$ to obtain a cocycle
\[
\iota_{E_k\times F_n} \colon E_k \times F_n \to |K(A^*(\pi_*(E\times F)_{\C}), k+n)|. 
\]

We can then form the diagram in $\sPre$  
\[
\xymatrix{
 & \Sing (E_k \times F_n) \ar[d]^-{\iota_{E_k\times F_n}} \\
\Sing |K(F^mA^*(\pi_*(E\times F)_{\C}), k+n)|  \ar[r] & \Sing |K(A^*(\pi_*(E\times F)_{\C}), k+n)|}
\]
The homotopy pullback of this diagram is the simplicial presheaf $(E_k\times F_n)(m)$.

\begin{lemma}\label{productlemma}
Let $E_k$ and $F_n$ be as above.  
Then, for any integer $m$, there is a canonical equivalence of Hodge filtered function spaces 
\[
(E_k \times F_n)(m) \xrightarrow{\approx} E_k(m) \times F_n(m).
\]
\end{lemma}
\begin{proof}
This follows from the fact that the singular functor and taking homotopy pullbacks commute with products and preserve weak equivalences. 
\end{proof}

\begin{prop}\label{productmapprop}
Let $E$, $F$ and $G$ be connective rationally even $\Omega$-spectra, and, for integers $k$, $n$ and $j$, let $f \colon E_k \to F_n \times G_j$ be a continuous map.  
Then, for any integer $m$, there is an induced commutative diagram 
\[
\xymatrix{
E_k(m) \ar[r] \ar[dr] & (F_n\times G_j)(m)  \ar[d]^{\approx} \\
 & F_n(m) \times G_j(m).}
\]
\end{prop}
\begin{proof}
The map $f$ consists of maps $E_k \to F_n$ and $E_k \to G_j$ which by Proposition \ref{inducedmap} induce maps $E_k(m) \to F_n(m)$ and $E_k(m) \to G_j(m)$. Together they give the map $E_k(m) \to F_n(m) \times G_j(m)$. That the diagram commutes follows from the fact that homotopy pullbacks respect products.  
\end{proof}


\section{Unstable splitting for Hodge filtered $BP$-spaces}



Let $p$ be a fixed prime number and $n$ a non-negative integer. Let $BP$ denote the $\Omega$-spectrum representing $BP$-cohomology at $p$ and let $BP\langle n \rangle$ be the $\Omega$-spectrum representing the $n$th intermediate theory defined in \cite{wilson}. 
We write $BP_k$ and $BP\langle n \rangle_k$ for the $k$th spaces of these spectra. 
For a given integer $m$, let $BP_k(m)$ and $BP\langle n \rangle_k(m)$ be the Hodge filtered function spaces associated to $BP_k$ and $BP\langle n \rangle_k$, respectively. 
Our main result is the following analog of Wilson's theorem \cite[Theorem 5.4]{wilson}.

\begin{theorem}\label{wilson5.4}
Let $m$ and $n$ be integers with $n\ge 0$.

{\rm a)} For every $k \leq 2(p^n + \cdots + 1)$, there is a weak equivalence of simplicial presheaves 
\bq\label{wilsonHodge5.4}
BP_k(m) \xrightarrow{\approx} BP\langle n \rangle _k(m) \times \prod_{j>n} BP\langle j \rangle_{k+2(p^j-1)}(m).
\eq

{\rm b)} For every $k \leq 2(p^{n-1} + \cdots + 1)$, there is a weak equivalence of simplicial presheaves 
\bq\label{wilsonHodge5.5}
BP_k\langle n \rangle(m) \xrightarrow{\approx} BP\langle n-1 \rangle _k(m) \times BP\langle n \rangle_{k+2(p^n-1)}(m).
\eq
\end{theorem}
\begin{proof}
a) In \cite[Corollary 3.6]{wilson}, Wilson shows that, for $k \leq 2(p^n + \cdots + 1)$, there is a map 
\begin{equation}\label{mainwilson}
BP_k \to BP\langle n \rangle_k \times \prod_{j>n} BP\langle j \rangle_{k+2(p^j-1)}.
\end{equation}

By Proposition \ref{productmapprop}, we obtain the map \eqref{wilsonHodge5.4} as the induced map of Hodge filtered function spaces. By \cite[Corollary 5.4]{wilson}, the map \eqref{mainwilson} is actually a homotopy equivalence. By the two-out-of-three property for weak equivalences, it follows from the commutative diagram of Proposition \ref{productmapprop} that the map  \eqref{wilsonHodge5.4} is a weak equivalence. 

b) The second equivalence is induced in the same way by the equivalence of \cite[Corollary 5.5]{wilson}. 
\end{proof}


Together with Theorem \ref{EDiso}, Theorem \ref{wilson5.4} implies the following result. 

\begin{cor}\label{wilson5.6}
Let $M$ be a complex manifold and $m$ and $n\ge 0$ be integers. Then the natural map 
\[
\BPD^k(m)(M) \to BP\langle n \rangle_{\Dh}^k(m)(M)
\]
is surjective for $k \leq 2(p^{n} + \cdots + 1)$.
\end{cor}

As in \cite[Theorem 5.7]{wilson}, we can use Corollary \ref{wilson5.6} to deduce an analog of Quillen's theorem. 
Recall \cite[p.105]{wilson} that, for given integers $k$, $n$ and $j:=2(p^n-1)$, there is a fiber sequence   
\[
BP\langle n \rangle_{k+j} \xrightarrow{v_n} BP\langle n \rangle_{k} \xrightarrow{f_n} BP\langle n-1 \rangle_{k}
\]
where $v_n$ denotes the $n$th generator in $BP^{*}=\Z_{(p)}[v_1, \ldots, v_n, \ldots]$. 
For every integer $m$, this sequence induces a commutative diagram in $\sPre$ 
\begin{equation}\label{1.4diag}
\scalebox{0.75}{
\xymatrix{
\Sing BP\langle n \rangle_{k+j} \ar[d] \ar[r] & \Sing BP\langle n \rangle_{k} \ar[d] \ar[r] & \Sing BP\langle n-1 \rangle_{k} \ar[d] \\
\Sing |K(A^*(\pi_*BP\langle n \rangle_{\C}, k+j))| \ar[r] & \Sing |K(A^*(\pi_*BP\langle n \rangle_{\C}, k))| \ar[r] & \Sing |K(A^*(\pi_*BP\langle n-1 \rangle_{\C}, k))| \\ 
\Sing |K(F^mA^*(\pi_*BP\langle n \rangle_{\C}, k+j))| \ar[u] \ar[r] & \Sing |K(F^mA^*(\pi_*BP\langle n \rangle_{\C}, k))| \ar[u] \ar[r] & \Sing |K(F^mA^*(\pi_*BP\langle n-1 \rangle_{\C}, k))| \ar[u]
}}
\end{equation}
in which each row is a fiber sequence (the functor $\Sing$ is right Quillen adjoint and preserves fiber sequences). Since homotopy pullbacks preserve fiber sequences, diagram \eqref{1.4diag} induces a fiber sequence in $\sPre$
\bq\label{1.4Hodge}
BP\langle n \rangle_{k+j}(m) \xrightarrow{v_n(m)} BP\langle n \rangle_{k}(m) \xrightarrow{f_n(m)} BP\langle n-1 \rangle_{k}(m).
\eq

Furthermore, there are the maps $g_n \colon BP_k \to BP\langle n \rangle_k$ which are compatible with $f_n$ in the sense that $f_n \circ g_n = g_{n-1}$. 
Each $g_n$ induces a map of simplicial presheaves
\[
g_n(m) \colon BP_k(m) \to BP\langle n \rangle_k(m).
\]

For every complex manifold $M$, the above maps induce a commutative diagram
\bq\label{diag118}
\xymatrix{
\BPD^{k+2(p^n-1)}(m)(M) \ar[d]_-{g_{n,m}} \ar[r]^-{v_{n,m}} & \BPD^{k}(m)(M) \ar[d]^-{g_{n,m}} & \\
BP\langle n \rangle_{\Dh}^{k+2(p^n-1)}(m)(M) \ar[r]^-{v_{n,m}} & BP\langle n \rangle_{\Dh}^{k}(m)(M) \ar[r]^-{f_{n,m}} & BP\langle n-1 \rangle_{\Dh}^{k}(m)(M)}
\eq
where the lower row is exact, since \eqref{1.4Hodge} is a fiber sequence.

Let $I^k\langle n \rangle(m)$ be the subgroup of elements in $\BPD^k(m)(M)$ which can be written as a finite sum 
\begin{align}\label{unform}
u=\sum_{i > n} v_{i,m}(u_i)
\end{align}
with $u_i \in \BPD^{k+2(p^i-1)}(m)(M)$ and $v_i \in BP^{-2(p^i-1)}$. 
Since the lower row of diagram \eqref{diag118} is exact, the subgroup $I^k\langle n \rangle(m)$ is contained in the kernel of the natural homomorphism 
\[
\BPD^k(m)(M) \to BP\langle n \rangle_{\Dh}^k(m) (M). 
\]

\begin{theorem}\label{injective} 
Let $M$ be a complex manifold and $m$ and $n \ge 0$ be integers. Then the induced homomorphism 
\[
\BPD^k(m)(M)/I^k\langle n \rangle(m) \to BP\langle n \rangle_{\Dh}^k(m) (M)
\]
is an isomorphism for $k \leq 2(p^n + \cdots + 1)$ and injective for $k \leq 2(p^n + \cdots + 1) + 2$. 
\end{theorem}
\begin{proof}
The surjectivity in dimensions $k \leq 2(p^n + \cdots + 1)$ follows directly from Corollary \ref{wilson5.6}. It remains to check injectivity the proof of which will follow as in \cite[Proof of Theorem 5.7]{wilson}.

Let $k \le 2(p^n + \cdots + 1) + 2$ and $u\in \BPD^k(m)(M)$ be an element which maps to $0$ in $BP\langle n \rangle_{\Dh}^k(m) (M)$. 
Since $M$ has the homotopy type of a finite complex, we know that the natural map is an isomorphism 
$BP^k(M) \cong BP\langle n \rangle^k(M)$ if $n$ is large enough. 
The map $g_n$ induces a morphism of long exact sequences 
\bq\label{BPles}
\scalebox{0.8}{
\xymatrix{
\cdots \ar[r] & BP_{\C}^{k-1}(M) \ar[r] \ar[d] & \BPD^k(m)(M)  \ar[r] \ar[d] & 
BP^k(M)\oplus H^k(M; F^mA^*(\pi_{*}BP_{\C})) \ar[r] \ar[d] & \cdots \\
\cdots \ar[r] & BP\langle n \rangle_{\C}^{k-1}(M) \ar[r] & \BPnD^k(m)(M) \ar[r] & 
BP\langle n \rangle^k(M)\oplus H^k(M; F^mA^*(\pi_{*}BP\langle n \rangle_{\C})) \ar[r] & \cdots}}
\eq

This shows that the associated Hodge filtered theories satisfy
\[
\BPD^k(m)(M) \cong BP\langle n \rangle_{\Dh}^k (m)(M) ~\text{for}~n~ \text{large enough}.
\] 

Hence we can find and fix the integer $q>n$ such that 
\begin{align*}
g_{q,m}(u)\ne 0 ~\text{and}~ f_{q,m}(g_{q,m}(u)) = g_{q-1,m}(u)=0.
\end{align*}
By the exactness of the lower row in diagram \eqref{diag118}, there is then an element 
\begin{align*}
u'\in BP\langle q \rangle_{\Dh}^{k+2(p^q-1)}(m)(M) ~ \text{with}~v_{q,m}(u')=g_{q,m}(u).
\end{align*}
Since $k \le 2(p^n + \cdots + 1) + 2$ and $q \ge n + 1$, we have
\[
k + 2(p^q-1) \le 2(p^n + \cdots + 1) + 2 + 2(p^q-1) \le 2(p^q + \cdots + 1).
\]

Hence, by Corollary \ref{wilson5.6}, there is an element $u_q$ in $\BPD^{k+2(p^q-1)}(m)(M)$ such that $g_{q,m}(u_q)=u'$ and $g_{q,m}(v_{q,m}(u_q))=g_{q,m}(u)$. 
Since $M$ has the homotopy type of a finite complex, it follows again from the upper long exact sequence in \eqref{BPles} that $\BPD^{k+2(p^j-1)}(m)(M)$ will be zero for $j$ large enough. Hence repeating this process with $u$ replaced by $u-v_{q,m}(u_q)$ shows that $u$ can be written as a finite sum of the form \eqref{unform} and lies in $I^k\langle n \rangle(m)$. 
\end{proof}


\begin{remark}
As mentioned in the introduction, we hope that Theorem \ref{injective} will help to find new interesting examples of algebraic cobordism cycles on smooth projective complex algebraic varieties. The idea is to study analogs of the Abel-Jacobi map of \cite{aj} for $BP\langle n \rangle$ for various $n$ in order to find families of cycles which are topologically trivial but non-trivial algebraically. Eventually, this should lead to new insights into the map from algebraic to complex cobordism for smooth complex varieties. 
\end{remark}


%
%
%
%
\bibliographystyle{amsalpha}

\end{document}